\documentclass[10pt,oneside]{article}
\usepackage[top=3cm,bottom=3cm,left=2cm,right=2cm,headheight=1cm,headsep=1cm,footskip=1cm]{geometry}
\usepackage{lipsum}
\usepackage[style,runin]{abstract}
\abslabeldelim{.}
\usepackage{titlesec}
\usepackage{comment}
\usepackage[dvipsnames]{xcolor}
\usepackage{ragged2e}
\usepackage[english]{babel}
\usepackage[T1]{fontenc}
\usepackage{amssymb}
\usepackage{amsmath}
\usepackage{amsfonts}
\usepackage{amsthm}
\usepackage{mathtools}
\usepackage{mathrsfs}
\usepackage{bbm}
\usepackage{accents}
\usepackage[usestackEOL]{stackengine}
\usepackage{scalerel}
\usepackage{tikz}
\usepackage[none]{hyphenat}
\usepackage{graphicx}
\usepackage{float}
\usepackage{multicol}
\usepackage[most]{tcolorbox}
\usepackage{enumitem}
\usepackage{url}
\usepackage{booktabs}
\usepackage{hyperref}
\hypersetup{
  colorlinks   = true, %Colours links instead of ugly boxes
  urlcolor     = blue, %Colour for external hyperlinks
  linkcolor    = blue, %Colour of internal links
  citecolor   = red %Colour of citations
}
\usepackage[ruled]{algorithm2e}
\usepackage{setspace}
\SetKwComment{Comment}{\string\\ }{}
\newcommand{\R}{\mathbb R}

\newcommand{\HH}{\mathcal H}

\newcommand{\N}{\mathbb N}

\newcommand{\RX}{\,]-\infty,+\infty]}
\newcommand{\RM}{\,]-\infty,0]}
\newcommand{\RPP}{\ensuremath{\left]0,+\infty\right[}}

\newcommand{\RP}{\ensuremath{\left[0,+\infty\right[}}
\newcommand{\menge}[2]{\left\{{#1}~\middle|~{#2}\right\}} 
\newcommand{\scal}[2]{{\left\langle{#1}~\middle|~{#2}\right\rangle}}

\newcommand{\ifaf}{\:\Leftrightarrow\:}
\newcommand{\norm}[1]{\left\|#1\right\|}
\newcommand{\persp}[1]{\widetilde{#1}}

\newcommand{\abs}[1]{\left| #1 \right|}

\def\overbar#1{\ThisStyle{%
  \setbox0=\hbox{$\SavedStyle#1$}%
  \stackengine{1.2\LMpt}{$\SavedStyle#1$}{\rule{\wd0}{.8\LMpt}}{O}{c}{F}{F}{S}%
}}
\newcommand{\cdom}{\overbar{\dom}\,}
\def\keywords#1{\par\addvspace\medskipamount{
\def\and{
	\ifhmode\unskip\nobreak\fi\ $\cdot$
}
\noindent\keywordname.\enspace\ignorespaces#1\par}
}
\def\keywordname{{\bfseries Keywords}}

\titleformat{\section}
  {\bfseries \Large}{\thesection}{1em}{}

\DeclareMathOperator{\epi}{epi}
\DeclareMathOperator{\dom}{dom}
\DeclareMathOperator*{\rec}{rec}
\DeclareMathOperator{\prox}{prox}

\DeclareMathOperator{\Id}{Id}

\DeclareMathOperator*{\argmin}{arg\,min}

\newcommand{\emp}{\ensuremath{{\varnothing}}}
\allowdisplaybreaks
\newtheorem{theorem}{Theorem}[section]
\newtheorem{proposition}{Proposition}[section]

\newtheorem{lemma}{Lemma}[section]
\theoremstyle{definition}
\newtheorem{definition}{Definition}[section]
\newtheorem{example}{Example}[section]
\newtheorem{remark}{Remark}[section]
\newtheorem{test}{Test}[section]

\numberwithin{equation}{section}
\title{\Large \textbf{Projection onto cones generated by epigraphs 
of perspective functions}}
\author{Luis M. Brice\~no-Arias$^1$, Crist\'obal Vivar-Vargas$^2$\\[.5cm]
\normalsize{$^1$Universidad T\'ecnica Federico Santa Mar\'ia, Departamento de Matem\'atica, Santiago, Chile} \\
\normalsize{\texttt{\href{mailto:luis.briceno@usm.cl}{luis.briceno@usm.cl}}} \\[2.5mm]
\normalsize{$^2$Universidad T\'ecnica Federico Santa Mar\'ia, Departamento de Matem\'atica, Santiago, Chile} \\
\normalsize{\texttt{\href{mailto:cristobal.vivar@usm.cl}{cristobal.vivar@usm.cl}}}}
\date{}
\begin{document}
\begin{titlepage}
	\centering
	\maketitle
	\vspace*{1cm}
\begin{abstract}
In this paper we provide an efficient computation of the projection 
onto the cone generated 
by the epigraph of the perspective of any convex lower 
semicontinuous function. 
Our formula requires solving only two scalar equations involving 
the proximity operator of the function. This enables the 
computation of projections, for instance, onto exponential and 
power cones, and extends to previously unexplored conic 
projections, such as the projection onto the hyperbolic
cone. We compare numerically 
the efficiency of the proposed approach in the case of 
exponential cones with an open source available method in the 
literature, 
illustrating its efficiency.
\end{abstract}

\end{titlepage}
%All acknowledgements should be placed in the back of the paper after Conclusions..
%%%%%%%%%%%%% Paper Body %%%%%%%%%%%%%%%%%%%%
\section{Introduction}
\label{s:1}
The perspective of a convex lower semicontinuous
function $f$ defined in a real Hilbert space $\HH$, denoted by 
$\widetilde{f}$, is a construction 
introduced in 
\cite{rockafellar1966level} and its epigraph turns out to be a closed 
convex cone.
The perspective appears naturally
in optimal mass transportation theory 
\cite{benamou2000computational,villani2021topics},
dynamical formulation of the 2-Wasserstein distance 
\cite{benamou2000computational,villani2021topics}, information 
theory \cite{elGheche2017proximity}, physics \cite{bercher2013}, 
operator theory \cite{effros2009}, statistics \cite{owen2007}, 
matrix 
analysis \cite{dacorogna2008},
signal processing and inverse problems 
\cite{kuroda2021convex,kuroda2022block,micchelli2013},
JKO \cite{jordan1998variational} schemes for gradient flows in 
the 
space of probability measures 
\cite{benamou2016augmented,carrillo2022primal}, and 
transportation and mean field games problems with 
state-dependent potentials
\cite{briceno2018proximal,cardaliaguet2015weak}, among other 
disciplines. 

In the particular case when 
$\HH=\R$, several known cones appearing in conic and 
mathematical programming problems are the 
epigraph of $\widetilde{f}$ for particular choices of $f$.
For instance, if  $f$ is the exponential function, this epigraph 
corresponds to 
the exponential cone $\mathcal{K}_{\exp}$, which appears in 
problems involving entropy 
functions, softmax 
and softplus activation functions from neural networks, and 
generalized posynomials 
in geometric programming 
\cite{mosek2022paper,friberg2023,Chares2009}. On the other 
hand, when $f$ is a power function, the epigraph of $\widetilde{f}$
is the well known power cone 
\cite{mosek2022paper,Chares2009,ECOS}.
Cones obtained as epigraphs of perspective functions appear 
naturally in conic reformulations of convex 
optimization problems (see Section~\ref{s:2} and, e.g., 
\cite{glineur2001conic}), but their projection has been only 
studied in particular cases. The projection onto the exponential 
cone has been recently developed in \cite{friberg2023}, which is 
used in known conic software packages 
\cite{scs2016,mosek2022paper,cosmo2021,ECOS}.

%\lb{Connection with conic formulation of convex optimization 
%problems, ADMM algorithm applied to quadratic functions 
%\cite{cosmo2021}, Douglas-Rachford method applied to 
%quadratic 
%functions \cite{scs2016}. Poner énfasis en que los softwares 
%implementan distintas rutinas para conos simétricos y solo 
%ciertos 
%conos no simétricos \cite{mosek2022paper}, nuestro método 
%agrega la posibilidad de computar la proyección sobre todos los 
%conos que puedan ser expresados como el epígrafo de alguna 
%función perspectiva (tuve una idea que no se si es plausible, 
%¿Se 
%podrán caracterizar como son los conos para los cuales existe 
%una función $f$ tal que el cono es igual a $\epi \widetilde{f}$? 
%¿Se podrá construir una función perspectiva para cualquier 
%cono?)}
% power cone, logarithmic, etc...

In this paper we provide an efficient computation of the projection 
onto the cone generated 
by the epigraph of the perspective of any convex lower 
semicontinuous function. 
Our formula involves the resolution of two scalar equations 
in which the
proximity operator of the function appears. We illustrate the 
efficiency of 
the proposed approach by comparing it with a state-of-the-art 
open-source
algorithm in the case of the exponential cone. Moreover, its 
flexibility is 
highlighted by providing the 
projection onto the epigraph of the 
perspective of an 
hyperbolic penalization function, which cannot be tackled by 
existing 
methods.

{\sloppy The manuscript is organized as follows. In Section~\ref{s:2},
we provide the notation and preliminaries on 
perspective functions, including a motivation in mathematical 
programming for the projection onto the epigraph of the
perspective. In Section~\ref{s:3} we exhibit our main result and a 
version for radial functions. In addition, we compute explicit 
formulae for 
the exponential cone and the hyperbolic cone and we 
give a bisection procedure for obtaining the 
projection to the epigraph of a perspective with explicit error 
bounds. Numerical comparisons in three 
different tests are 
provided in Section~\ref{s:4}.

}

\section{Notation and preliminaries}
\label{s:2}
Throughout this paper, $\mathcal{H}$ is a real Hilbert space 
endowed with the inner product $\scal{\cdot}{\cdot}$ and 
associated norm $\norm{\cdot}$. $\HH \oplus \R$ denotes the 
Hilbert direct sum between $\HH$ and $\R$.

Given $f: \HH \to \RX$, the domain of $f$ is $\dom f = \menge{x \in 
\HH} {f(x) < +\infty}$ and $f$ is proper if $\dom f\neq \emp$. 
Denote by $\Gamma_0(\mathcal{H})$ the class of proper lower semicontinuous convex functions from 
$\HH$ to $]-\infty, +\infty]$ and let $f \in \Gamma_0(\HH)$. 
The recession function of $f$ is
\begin{equation}
	\label{e:recession}
	(\forall x_0 \in \dom f) (\forall x \in \HH)\quad (\rec f)(x) = \lim_{t 
	\to +\infty} \frac{f(x_0 + t x) - f(x_0)}{t}
\end{equation}
and its perspective, which has a central role in this manuscript, is 
introduced in the following definition.
\begin{definition}\label{def:perspective} Let $f \in 
\Gamma_0(\HH)$. The \textit{perspective} of $f$ is:
\begin{equation}
	\label{e:persp_def}
	\widetilde{f}\colon\HH\times\R\to\RX\colon
	(x,\eta) \mapsto 
	\begin{cases}
		\eta f \left(\dfrac{x}{\eta}\right),&\text{if}\:\:\eta > 0; \\[3mm]
		(\rec f) (x), & \text{if}\:\:\eta = 0; \\[2mm]
		+\infty, & \text{if}\:\:\eta < 0.
	\end{cases}    
\end{equation}
\end{definition} 
The conjugate of $f$ is 
\begin{equation}
	\label{e:conjugate}
	f^*: \HH \to \RX: u \mapsto \sup_{x \in \HH} \left(\scal{x}{u} - f(x)\right).
\end{equation} 
We have $f^*\in\Gamma_0(\HH)$, $f^{**} = f$, and 
\begin{equation}
	\label{e:fenchel_young}
	(\forall x \in \HH)(\forall u \in \HH) \quad f(x) + f^*(y) \geq 
	\scal{x}{u},
\end{equation}
known as the 
Fenchel-Young inequality 
\cite[Proposition~13.15]{bauschke2011convex}.
The function $f$ is supercoercive if 
$$\lim_{\|x\|\to+\infty}\frac{f(x)}{\|x\|}=+\infty,$$ in which case 
$\dom 
f^*=\HH$ \cite[Proposition~14.15]{bauschke2011convex}.
The subdifferential of $f$ is the set-valued operator
\begin{equation}
	\label{e:subdifferential}
	\partial f: \HH \to 2^{\HH} : x \mapsto 
	\menge{u \in \HH}{(\forall y \in \HH) \quad \scal{y-x}{u} + f(x) \leq 
	f(y)},
\end{equation}  
which satisfies
the Fenchel-Young identity 
\cite[Proposition~16.10]{bauschke2011convex}
\begin{equation}
	\label{e:fenchel_young_iden}
	(\forall x \in \HH)(\forall u \in \HH) \quad u \in \partial f(x) 
	\quad\Leftrightarrow\quad f(x) + f^*(u) = \scal{x}{u},
\end{equation}
and $\dom \partial f=\menge{x\in\HH}{\partial f(x)\neq\emp}$.
The \textit{proximity operator} of $f$ is
\begin{equation}
	\label{e:prox_def}
	\prox_f : \HH \to \HH : x \mapsto \argmin_{y \in \HH} \left(f(y) + \frac{1}{2} \norm{x-y}^2\right),
\end{equation}
which is characterized by
\begin{equation}
	\label{e:prox_char}
	(\forall x \in \HH)(\forall p \in \HH) \quad p = \prox_f x \quad\Leftrightarrow\quad x-p \in \partial f(p)
\end{equation}
and satisfies \cite[Proposition~24.8]{bauschke2011convex}
\begin{equation}
	\label{e:moreau_decomp}
	(\forall\gamma \in\RPP)\quad  \prox_{\gamma f} =\Id - \gamma \prox_{f^*/\gamma}\circ (\Id/\gamma),
\end{equation}
where $\Id\colon\HH\to\HH$ denotes the identity operator.

Let $C\subset\HH$ be a nonempty closed convex set. The indicator function of $C$ is
\begin{equation}
	\label{e:indicator}
	\iota_C: \HH \to \RX: x \mapsto 
	\begin{cases}
		0, &  \text{if }x \in C;\\
		+\infty, &\text{if }x \not\in C,    
	\end{cases}   
\end{equation}
its support function is
\begin{equation}
	\label{e:support}
	\sigma_C : \HH \to\RX: u \mapsto\sup_{x \in C} \scal{x}{u}, 
\end{equation}
we have $\sigma_C= (\iota_C)^*$, and the projection operator 
onto $C$ is $P_C=\prox_{\iota_C}$.
For further background on convex analysis, the reader is referred 
to \cite{bauschke2011convex}. 
Note that \cite[Proposition~7.13 \& 
Proposition~13.49]{bauschke2011convex} imply
\begin{equation}
	\label{e:recession_char}
	\rec{f} = \sigma_{\dom f^*} = \sigma_{\;\cdom f^*}.
\end{equation}

The following lemma is useful in the construction of synthetic data 
for the numerical tests in Section~\ref{s:4}.
\begin{lemma}
	\label{l:normal}
	Let $f \in \Gamma_0(\HH)$ and let $x\in \dom \partial f$. Set
	\begin{equation}
		\label{e:normal}
		(\tilde{x},\tilde{\delta}) \in (x,f(x)) + \RPP\cdot
		(\partial f(x) \times \{-1\}).
	\end{equation}
	Then $(\tilde{x},\tilde{\delta})\notin\epi f$ and 
	$P_{\epi f}(\tilde{x},\tilde{\delta}) = 
	(x,f(x))$.
\end{lemma}
\begin{proof}
Clear 
	from \cite[Proposition~16.16]{bauschke2011convex} and 
	\cite[Proposition~6.47]{bauschke2011convex}.
\end{proof}
%\begin{remark}
%	\label{r:f_diff}
%	In the case that $(x,\eta,\delta) \in \dom f \times \RPP\times 
%\R$ 
%	for a differentiable $f \in \Gamma_0(\HH)$, 
%	Lemma~\ref{l:normal} reduces to the fact that the projection 
%	onto $\epi\widetilde{f}$ of any element of the outward pointing 
%	normal line of $\epi\widetilde{f}$ at $(x,\eta,\delta)$ is in fact 
%	$(x,\eta,\delta)$.
%\end{remark}

%% Old Model
%	\begin{equation}
%	\label{e:model}
%	f \in \Gamma_0(\HH);\quad \begin{array}{ll}
%		\min f(x) + r\theta(f_i(x)/r)
%	\end{array}
%\end{equation}

\subsection{Motivation}
Let $n \in \N$, let $\{f_i\}_{i=0}^n \subset \Gamma_0(\HH)$, and 
consider the constrained mathematical programming problem
\begin{equation}
	\label{e:model}
	\min_{x \in C} f_0(x),
\end{equation}
where $C = \menge{x \in \HH}{(\forall i \in \{1,\ldots,n\})\:\: f_i(x)\le 
0}$.
In order to solve these type of problems, a class of epigraphical 
first-order methods (see, e.g., 
\cite{chierchia2015epigraphical,tofighi2014denoising,tofighi2014denoising,wang2016epigraph})
 use a sequence of projections onto the epigraphs of the functions 
$\{f_i\}_{i=0}^n$ by reformulating \eqref{e:model} equivalently as
\begin{equation}
	\label{e:model_epi}
	\min_{(x,\delta) \in\,\epi f_0\, \cap\, (C\times \R)} \delta.
\end{equation}
Moreover, in conic optimization \cite{glineur2001conic}, the 
constraint in \eqref{e:model_epi} is relaxed to its closed conical 
hull, which is strongly related to the epigraphs of the 
perspectives (see \eqref{e:persp_def}) of $\{f_i\}_{i = 0}^n$. 
Indeed, note that
\begin{align}
	\label{e:model_cone}
	\RP\cdot(\epi f_0 \cap C\times \R) & =  \bigcup_{\eta \in 
	\RPP}\menge{(x,\delta) \in \HH\times \R}{(x/\eta, \delta/\eta) \in 
	\epi f_0 \cap C\times \R} \nonumber \\
	& = \bigcup_{\eta \in \RPP}\menge{(x,\delta) \in \HH\times 
	\R}{ \mathcal{P}_{f_0}(x,\eta) \le \delta,\:(\forall 
	i \in \{1,\ldots,n\})\:\:  \mathcal{P}_{f_i}(x,\eta)\le 0},
\end{align}
where
\begin{equation}
	(\forall i \in \{1,\ldots,n\}) \quad \mathcal{P}_{f_i}: \HH\times \R 
	\to \RX :(x,\eta) \mapsto\begin{cases}
		\eta f_i\left(\dfrac{x}{\eta}\right), &\text{if}\:\: \eta > 0;\\
		+\infty, &\text{if}\:\: \eta \le 0.
	\end{cases} 
\end{equation}
Furthermore, \cite{rockafellar1966level} asserts that, for every 
$i\in\{0,\ldots,n\}$, $\overline{\epi
\mathcal{P}_{f_i}}=\epi\widetilde{f_i}$, which implies
\begin{equation}
	\overline{\RP\cdot(\epi f_0 \cap C\times \R)} = \bigcup_{\eta \in 
	\RPP}\menge{(x,\delta) \in \HH\times \R}{(x,\eta,\delta) \in \epi 
	\widetilde{f_0} \:\: \text{and}\:\: (x,\eta,0) \in \bigcap_{i=1}^n \epi 
	\widetilde{f_i}}.
\end{equation}
This motivates several first order approaches 
(see, e.g., 
\cite{cosmo2021,scs2016,ceria1999convex,zhang2004new}) 
using
the computation of the projections 
$P_{\epi\widetilde{f_0}},\ldots,P_{\epi\widetilde{f_n}}$ to solve 
the conic formulation \eqref{e:model_epi}.

\subsection{Perspective functions and properties}
Now, we review essential properties of perspective functions. We 
refer the reader to \cite{combettes2018perspective} for further 
background.
\begin{lemma}
\label{l:persp_prop}
Let $f \in \Gamma_0(\HH)$. Then the following hold:
\begin{enumerate}[label = \normalfont(\roman*)]
	\item
	\label{l:persp_prop_i}
	$\persp{f} \in \Gamma_{0}(\HH\oplus \R)$.
%	\item
%	\label{l:persp_prop_ii}
%	$\persp{f}$ is not radial, i.e., there exist $(x,\eta)$ and 
%$(y,\nu)$ in $\HH\times\R$ such that $\|(x,\eta)\|=\|(y,\nu)\|$ and 
%$\persp{f}(x,\eta)\ne\persp{f}(y,\nu)$.
	\item
	\label{l:persp_prop_iii}
	Let $C = \menge{(x,\eta) \in \HH \times \R}{\eta + f^*(x) \leq 0}$. Then $\left(\persp{f}\right)^* = \iota_C$. 
	\item 
	\label{l:persp_prop_iv}
	$\epi \widetilde{f}$ is a closed convex cone, i.e., $\epi 
	\widetilde{f} = \RP\cdot 
	\epi 
	\widetilde{f}$.
\end{enumerate}
\end{lemma}
\begin{proof}
	\ref{l:persp_prop_i}: \cite[Proposition~2.3(ii)]{combettes2018perspective}.
%	
%	\ref{l:persp_prop_ii}: Let $(x,\eta) \in\HH\times\RPP$ be such 
%that $x/\eta \in \dom f$. Then $\norm{(x,\eta)} = \norm{(x,-\eta)}$ 
%and $\persp{f}(x,\eta) = \eta f(x/\eta) < +\infty = 
%\persp{f}(x,-\eta)$. 
%Hence $f$ is not radial. 
	
	\ref{l:persp_prop_iii}: 
	\cite[Proposition~2.3(iv)]{combettes2018perspective}.
	
	\ref{l:persp_prop_iv} In view of \ref{l:persp_prop_i} and 
	\cite[Proposition~2.3(i)]{combettes2018perspective} the claim 
	follows from \cite[Proposition~6.2]{bauschke2011convex}.
\end{proof}
The following result is a slight modification of 
\cite[Theorem~3.1]{bricenoarias2023enhanced} and it is crucial
for our main result.
\begin{proposition}
	\label{l:prox_persp}
	Let $f \in \Gamma_0(\HH)$, let $\gamma \in \RPP$, and let 
	$(x,\eta) \in \HH \times \R$. Then the following hold:
	\begin{enumerate}[label = \normalfont(\roman*)]
		\item 
		\label{l:prox_persp_i}
		Suppose that $\eta + \gamma f^* \big(P_{\;\cdom f^*} \, 
		(x/\gamma)\big) \leq 0$. Then
		\begin{equation}
			\label{e:prox_persp_i}
			\prox_{\gamma \widetilde{f}}(x,\eta) = \left(x - \gamma 
			P_{\; \cdom f^*} \left(\frac{x}{\gamma}\right), 0\right). 
		\end{equation} 
		\item 
		\label{l:prox_persp_ii}   
		Suppose that $\eta + \gamma f^* \big(P_{\;\cdom f^*} 
		\,(x/\gamma)\big) > 0$. Then there exists a unique 
		$\mu\in\big]0,\eta + \gamma f^* \big(P_{\;\cdom f^*} 
		\,(x/\gamma)\big)\big]\cap\R$ such that
		\begin{equation}
			\label{e:prox_persp_iimu}
			\mu = \eta + \gamma f^* 
			\left(\prox_{\frac{\mu}{\gamma}f^*}\left(\frac{x}{\gamma}\right)\right).
		\end{equation} 
		Furthermore
		\begin{equation}
			\label{e:prox_persp_ii}
			\prox_{\gamma \widetilde{f}}(x,\eta) = \left(\mu 
			\prox_{\frac{\gamma}{\mu} f} 
			\left(\frac{x}{\mu}\right),\mu\right).
		\end{equation}
	\end{enumerate}
\end{proposition}
%%%%%%%%%%% Old Lemma Rocki %%%%%%%%%%%%%%%%%
%\cv{The following lemma is a direct consequence of 
%\cite[Theorem~3.1]{bricenoarias2023enhanced}.
%\begin{lemma}
%\label{l:prox_perps_rocki}
%Let $f \in \Gamma_0(\HH)$, denote
%\begin{equation}
%	\label{e:rocki}
%	\left(\forall \xi \in \RP\right) \quad \xi {\odot} f = \begin{cases}
%		\iota_{\cdom f}, &\text{if}\:\:\xi = 0;\\
%		\xi f, &\text{if}\:\:\xi > 0,
%	\end{cases}
%\end{equation}let $\gamma \in \RPP$, and let $(x,\eta) \in \HH 
%\times \R$. Then there exists a unique $\mu \in \big[0,\eta + 
%\gamma f^* \big(P_{\;\cdom f^*} \,(x/\gamma)\big)\big]$ such that
%\begin{equation}
%	\label{e:mu_rocki}
%	\mu = \max\left\{\eta + \gamma f^*\left(\prox_{\frac{\mu}{\gamma} {\odot} f^*} \left(\frac{x}{\gamma}\right)\right),0\right\}.
%\end{equation}Furthermore
%\begin{equation}
%	\label{e:prox_rocki}
%	\prox_{\gamma \widetilde{f}} (x,\eta) = \left(x- \gamma \prox_{\frac{\mu}{\gamma} {\odot} f^*} \left(\frac{x}{\gamma}\right) , \mu \right).
%\end{equation}
%\end{lemma}
%\begin{proof}
%	Combine \eqref{e:rocki}  and
%	\cite[Theorem~3.1]{bricenoarias2023enhanced}.
%\end{proof}}

The following result is a refinement of Proposition~\ref{l:prox_persp} 
in the case 
when the function is radial and is also a slight modification of 
\cite[Proposition~3.3]{bricenoarias2023enhanced}.
\begin{proposition}
\label{pro:plc_radial} 
Let $\varphi \in \Gamma_0(\R)$ be even, 
set $f = \varphi \circ \norm{\cdot}$, let $\gamma \in \RPP$, and let 
$(x,\eta) \in \HH \times \R$. Then the following hold:
    \begin{enumerate}[label = \normalfont(\roman*)]
\item 
\label{pro:plc_radial_i}
Suppose that $\eta + \gamma \varphi^* \big(P_{\;\cdom 
\varphi^*}\, 
(\norm{x}/\gamma)\big) \leq 0$. Then
    \begin{equation}
        \prox_{\gamma \widetilde{f}}(x,\eta) = 
        \begin{cases}
        \left(\left(1 - \gamma \dfrac{P_{\; \cdom \varphi^*} 
        \left(\|x\|/\gamma\right)}{\|x\|}\right)x, 0\right),&\text{if}\:\:x\neq 0;\\[2mm]
        (0,0),&\text{if}\:\:x=0.
        \end{cases} \label{eq:plc_radial_leq0}
    \end{equation} 
\item 
\label{pro:plc_radial_ii}
    Suppose that $\eta + \gamma \varphi^* \big(P_{\;\cdom 
    \varphi^*}\,  
    (\norm{x}/\gamma)\big) > 0$. Then there exists a unique 
    $\mu\in\big]0,\eta + \gamma \varphi^* \big(P_{\;\cdom 
    \varphi^*}\, 
(\norm{x}/\gamma)\big)\big]\cap\R$ such that
\begin{equation}
\label{eq:plc_radial_geq0_mu}
\mu = \eta + \gamma \varphi^* 
\bigg(\prox_{\frac{\mu}{\gamma}\varphi^*}\bigg(\frac{\|x\|}{\gamma}\bigg)\bigg).
\end{equation} 
Furthermore
    \begin{equation}
\label{eq:plc_radial_geq0_prox}
        \prox_{\gamma \widetilde{f}}(x,\eta) = 
        \begin{cases}
            \left(\prox_{\frac{\gamma}{\mu}\varphi}\big(\frac{\|x\|}{\mu}\big)
            \frac{\mu x}{\|x\|} ,\mu\right),&\text{if}\:\:x\neq 0;\\[2mm]
            (0,\eta+\gamma\varphi^*(0)), &\text{if}\:\:x= 0.
        \end{cases}
    \end{equation}
    \end{enumerate}    
\end{proposition} 

\section{Main results}
\label{s:3}
Now we provide our main result, which provides an explicit 
formula for the projection onto the epigraph of a perspective 
function via its proximity operator. 

\begin{theorem}
\label{t:main_result} 
Let $f \in \Gamma_0(\HH)$ and let 
$(x,\eta,\delta)\in\HH\times\R^2$. Then we have
\begin{equation}
	\label{e:P_epi_persp}
P_{\epi \widetilde{f}}(x,\eta,\delta)=
\begin{cases}
\big(P_{\,\cdom \widetilde{f}}(x,\eta),\delta\big),&\text{if}\:\: 
\widetilde{f}\big(P_{\, \cdom \widetilde{f}} (x,\eta)\big) \le 
\delta;\\[2mm]
 \big(\prox_{\mu\widetilde{f}}(x,\eta),\delta + 
 \mu\big),&\text{if}\:\:\widetilde{f}
\big(P_{\, \cdom \widetilde{f}} (x,\eta)\big) > \delta,
\end{cases} 
\end{equation}
where $\mu\in\,\big]0,-\delta + 
	\widetilde{f}(P_{\,\cdom \widetilde{f}}(x,\eta))\big]\cap\R$ is the 
	unique 
	solution to 
\begin{equation}
	\label{e:eqmain}
	\mu +\delta 
	-\widetilde{f}\left(\prox_{\mu\widetilde{f}}(x,\eta)\right)=0.	
\end{equation}
\end{theorem}
\begin{proof}
Set $L : \HH \oplus \R^2 \to \HH \oplus \R^2: 
(x,\eta,\delta) \mapsto (x,\eta,-\delta)$ and denote
\begin{equation}
	\widetilde{C} \coloneqq \menge{(x,\eta,\delta) \in \HH \times 
	\R^2}{\delta + 
	\widetilde{f}(x,\eta) \le 0}.
\end{equation}
Since $L = L^{-1} = L^*$ and $\epi \widetilde{f} = L(\widetilde{C})$, 
it follows from \cite[Proposition~29.2(ii)]{bauschke2011convex} that
\begin{equation}
	\label{e:P_epi}
	P_{\epi \widetilde{f}} = P_{L(\widetilde{C})} = L \circ 
	P_{\widetilde{C}} \circ L.
\end{equation}
Furthermore, it follows from 
Lemma~\ref{l:persp_prop}\ref{l:persp_prop_i} that 
$\widetilde{f} \in 
\Gamma_0(\HH \oplus \R)$, which yields $\big(\widetilde{f}\big)^{*} 
\in 
\Gamma_0(\HH \oplus \R)$ and 
\begin{equation}
\label{e:dobcon}
\Big(\widetilde{f}\Big)^{**}=\widetilde{f}.
\end{equation}
Hence, 
Lemma~\ref{l:persp_prop}\ref{l:persp_prop_iii} and
$\eqref{e:moreau_decomp}$
yield
\begin{equation}
	\label{e:PC}
	P_{\widetilde{C}} =\prox_{\iota_{\widetilde{C}}} = \Id - 
	\prox_{\widetilde{\big(\widetilde{f}\big)^*}}.
\end{equation}
Therefore, \eqref{e:P_epi} and \eqref{e:PC} imply that
\begin{equation}
\label{e:auxmain}
P_{\epi 
\widetilde{f}}(x,\eta,\delta) 
=L\left(\Id - 
\prox_{\widetilde{\big(\widetilde{f}\big)^*}}\right)(x,\eta,-\delta) = (x,\eta,\delta) - L 
	\left(\prox_{\widetilde{\big(\widetilde{f}\big)^*}}(x,\eta,-\delta)\right).
\end{equation}
Hence, in order to compute $P_{\epi 
\widetilde{f}}$ we consider Proposition~\ref{l:prox_persp} with
the function
$\big(\widetilde{f}\big)^*$ and 
$\gamma=1$. We consider two 
cases.
\begin{enumerate}
\item $\underline{\widetilde{f}(P_{\,\cdom 
\widetilde{f}}(x,\eta)) \le \delta}$:
It follows from 
Proposition~\ref{l:prox_persp}\ref{l:prox_persp_i} and 
\eqref{e:dobcon} that
\begin{align}
\prox_{\widetilde{\big(\widetilde{f}\big)^*}}(x,\eta,-\delta)
=\left((x,\eta)-P_{\cdom\widetilde{f}}(x,\eta),0\right),
\end{align}
and \eqref{e:auxmain} reduces to
\begin{equation}
P_{\epi \widetilde{f}}(x,\eta,\delta)= (x,\eta,\delta) - 
		\left((x,\eta) - P_{\, \cdom \widetilde{f}}(x,\eta),0\right) = 
		\left(P_{\, \cdom \widetilde{f}}(x,\eta),\delta\right).
\end{equation}

\item $\underline{\widetilde{f}(P_{\,\cdom 
\widetilde{f}}(x,\eta)) > \delta}$:
	Proposition~\ref{l:prox_persp}\ref{l:prox_persp_ii} and 
\eqref{e:dobcon} imply that there 
	exists a unique $\mu \in ~]0,-\delta + \widetilde{f}(P_{\,\cdom 
		\widetilde{f}}(x,\eta))]$ such that $\mu =-\delta+ 
		\widetilde{f}(\prox_{\mu \widetilde{f}\,}(x,\eta))$ and 
\begin{align}
\prox_{\widetilde{\big(\widetilde{f}\big)^*}}(x,\eta,-\delta)
=\left((x,\eta)-\prox_{\mu 
\widetilde{f}}(x,\eta),\mu\right).
\end{align}
Therefore, \eqref{e:auxmain} reduces to
	\begin{align}
		\label{e:MR_aux_ge1}
		P_{\epi \widetilde{f}}(x,\eta,\delta) &= (x,\eta,\delta) - 
		L\left((x,\eta) - \prox_{\mu \widetilde{f}}(x,\eta), 
		\mu\right)\nonumber\\
&= (x,\eta,\delta) - 
		\left((x,\eta) - \prox_{\mu \widetilde{f}}(x,\eta), 
		-\mu\right)\nonumber\\
& = 
\left(\prox_{\mu \widetilde{f}}(x,\eta), \delta +\mu\right).
	\end{align}
\end{enumerate}
The proof is complete.
\end{proof}

\begin{remark}
	\label{r:phi}
\begin{enumerate}[label=(\roman*)]
\item Note that the first condition of \eqref{e:P_epi_persp} in 
Theorem~\ref{t:main_result} asserts that, if $\big(P_{\,\cdom 
\widetilde{f}}(x,\eta),\delta\big)\in\epi \widetilde{f}$ then $P_{\epi 
\widetilde{f}}(x,\eta,\delta)=\big(P_{\,\cdom 
\widetilde{f}}(x,\eta),\delta\big)$, which is an explicit expression 
for 
$P_{\epi\widetilde{f}}(x,\eta,\delta)$ for points $(x,\eta)$ which are 
not necessarily in the domain of $\widetilde{f}$.
\item
\label{r:phi_prop} 
In the context of Theorem~\ref{t:main_result}, suppose that 
$\widetilde{f}(P_{\,\cdom 
\widetilde{f}}(x,\eta))>\delta$ and set $\phi\colon 
\mu\mapsto \mu +\delta 
-\widetilde{f}(\prox_{\mu\widetilde{f}}(x,\eta))$. We deduce from 
\cite[Lemma~3.27]{attouch1984variational} and 
\cite[Lemma~3.3]{bricenoarias2023proximity} that $\phi$ is 
continuous, strictly increasing in $\RPP$, 
$\lim_{\mu\to+\infty}\phi(\mu)=+\infty$, and $\lim_{\mu\downarrow 
0}\phi(\mu)=\delta-\widetilde{f}(P_{\,\cdom 
\widetilde{f}}(x,\eta))<0$. Therefore, the unique 
solution to $\phi(\mu)=0$ guaranteed by 
Theorem~\ref{t:main_result} can be obtained by state-of-the-art 
root finding algorithms \cite{press2007numerical}.
\end{enumerate}
\end{remark}
The following result gives a closed form expression for 
$P_{\epi 
\widetilde{f}}(x,\eta,\delta)$ in
Theorem~\ref{t:main_result}.

\begin{proposition}
\label{p:f_prox}
Let $f\in\Gamma_0(\HH)$, let $\mu\in\RPP$, and let 
$(x,\eta)\in\HH\times\R$. Then
\begin{equation}
\prox_{\mu \widetilde{f}}(x,\eta)
=
\begin{cases}
\big(\prox_{\mu(\rec f)}x,0\big),&\text{if}\:\:\eta+\mu 
f^*\Big(P_{\,\cdom f^*}\big(\frac{x}{\mu}\big)\Big)\le 0;\\[3mm]
\Big(\nu\prox_{\frac{\mu}{\nu} 
f}\big(\frac{x}{\nu}\big),\nu\Big),&\text{if}\:\:\eta+\mu 
f^*\Big(P_{\,\cdom f^*}\big(\frac{x}{\mu}\big)\Big)> 0,
\end{cases}
\end{equation}
and 
\begin{equation}
\label{e:fproxf}
\widetilde{f}\big(\prox_{\mu 
\widetilde{f}}(x,\eta)\big)
=
\begin{cases}
(\rec f)\big(\prox_{\mu(\rec f)}x\big),&\text{if}\:\:\eta+\mu 
f^*\Big(P_{\,\cdom f^*}\big(\frac{x}{\mu}\big)\Big)\le 0;\\[3mm]
\nu f\Big(\prox_{\frac{\mu}{\nu} 
f}(\frac{x}{\nu})\Big),&\text{if}\:\:\eta+\mu 
f^*\Big(P_{\,\cdom f^*}\big(\frac{x}{\mu}\big)\Big)> 0,
\end{cases}
\end{equation}
where $\nu\in\big]0,\eta + \mu f^*(P_{\cdom \, f^*} 
(x/\mu))\big]\cap\R$ is the unique solution to 
\begin{equation}
\label{e:nu}
\nu=\eta+\mu 
f^*\left(\prox_{\frac{\nu}{\mu}f^*}\left(\frac{x}{\mu}\right)\right).
\end{equation}
\end{proposition} 
\begin{proof}
In view of Proposition~\ref{l:prox_persp}, we explore two cases.
If $\eta + \mu f^*(P_{\cdom \, f^*} (x/\mu)) \le 0$, it follows 
from Proposition~\ref{l:prox_persp}\ref{l:prox_persp_i},
\eqref{e:moreau_decomp}, and \eqref{e:recession_char} that
\begin{equation}
\prox_{\mu \widetilde{f}}(x,\eta)
=\left(\prox_{\mu \sigma_{\cdom f^*}} x,0\right)=\left(\prox_{\mu 
(\rec f)} x,0\right),
\end{equation}
and \eqref{e:persp_def} yields
	\begin{align}
		\widetilde{f}\left(\prox_{\mu \widetilde{f}}(x,\eta)\right) 
= (\rec f)\left(\prox_{\mu (\rec f)} x\right).
	\end{align}

On the other hand, if $\eta + \mu f^*(P_{\cdom \, f^*} 
(x/\mu)) > 0$, Proposition~\ref{l:prox_persp}\ref{l:prox_persp_ii}
asserts that
there exists a unique $\nu \in\big]0,\eta + \mu f^*(P_{\cdom \, 
f^*} 
(x/\mu))\big]\cap\R$ such that 
$\nu=\eta+\mu f^*(\prox_{\nu f^* /\mu}(x/\mu))$ 
and
\begin{align}
\prox_{\mu \widetilde{f}}(x,\eta)
&=\left(\nu \prox_{\frac{\mu}{\nu} f} 
\left(\frac{x}{\nu}\right),\nu\right).
\end{align}
Hence, it follows from \eqref{e:persp_def} that
	\begin{align}
		\widetilde{f}\left(\prox_{\mu \widetilde{f}}(x,\eta)\right) &= 	
		\widetilde{f}\left(\prox_{\frac{\mu}{\nu} f} 
\left(\frac{x}{\nu}\right),\nu\right)
	\end{align}
and the result follows.
\end{proof}
The next result specifies Proposition~\ref{p:f_prox} for the particular case of radial functions.
\begin{proposition}
	\label{p:fproxf_radial}
Let $\varphi \in \Gamma_0(\R)$ be an even supercoercive 
function such that
$\dom\varphi=\R$, set $f = \varphi \circ 
	\norm{\cdot}$, and let $(x,\eta,\delta) \in \HH 
	\times 
	\R\times\R$. Then we have
\begin{equation}
P_{\epi \widetilde{f}}(x,\eta,\delta)=
\begin{cases}
(x,\max\{\eta,0\},\delta),&\text{if}\:\:
\widetilde{\varphi}(\|x\|,\max\{\eta,0\})\le \delta;\\[1mm]
(0,0,0),&\text{if}\:\:
\widetilde{\varphi}(\|x\|,\max\{\eta,0\})> \delta,\:\:
\delta<0,\:\:\text{and}\:\:\eta-\delta\varphi^*(\frac{\|x\|}{-\delta})\le 
0;\\[1mm]
\left(0,\frac{\eta-\delta\varphi^*(0)}{1+(\varphi^*(0))^2},
\frac{\varphi^*(0)(\delta\varphi^*(0)-\eta)}{1+(\varphi^*(0))^2}\right),
&\text{if}\:\:\widetilde{\varphi}(\|x\|,\max\{\eta,0\})> 
\delta,\:\:\:\:\eta-\delta\varphi^*(0)>0,
\:\:\text{and}\:\:x= 0;\\[2mm]
\left(\prox_{\frac{\mu}{\nu}\varphi}
\left(\frac{\|x\|}{\nu}\right)\frac{\nu 
x}{\|x\|},\nu,\delta+\mu\right),&\text{otherwise},
\end{cases}
\end{equation}
	where $\nu>0$ and $\mu>0$ are the unique solutions to the 
	 $2\times 2$ system of nonlinear equations
	\begin{eqnarray}
		\label{e:mu_radial}
\mu+\delta&=&\nu\varphi\left(\prox_{\frac{\mu}{\nu}\varphi} 
		\left(\frac{\norm{x}}{\nu}\right)\right)\\
		\nu-\eta&=&\mu \varphi^*\left(\prox_{\frac{\nu}{\mu}\varphi^*} 
		\left(\frac{\norm{x}}{\mu}\right)\right).		\label{e:nu_radial}
	\end{eqnarray}
\end{proposition}
\begin{proof}
Since $\dom\varphi=\R$, it follows from 
Definition~\ref{e:persp_def} that $\cdom 
\widetilde{f}=\HH\times\RP$, which yields 
$P_{\cdom 
\widetilde{f}}\colon (x,\eta)\mapsto (x,\max\{\eta,0\})$ and
\begin{equation}
\label{e:ftradial}
\widetilde{f}\big(P_{\cdom 
\widetilde{f}}(x,\eta)\big)
=\widetilde{\varphi}\big(\|x\|,\max\{\eta,0\}\big).
\end{equation}
We split the proof in two cases:

\underline{$\widetilde{\varphi}(\|x\|,\max\{\eta,0\})\le\delta$}:
We deduce from Theorem~\ref{t:main_result} that
$P_{\epi \widetilde{f}}(x,\eta,\delta)=
(x,\max\{\eta,0\},\delta)$.

\underline{$\widetilde{\varphi}(\|x\|,\max\{\eta,0\})>\delta$}:
Theorem~\ref{t:main_result} implies that 
$P_{\epi 
\widetilde{f}}(x,\eta,\delta)
=(\prox_{\mu\widetilde{f}}(x,\eta),\delta+\mu)$, where 
$\mu>0$ is the unique solution to 
\begin{equation}
\label{e:eqmurad}
\mu +\delta 
	-\widetilde{f}\left(\prox_{\mu\widetilde{f}}(x,\eta)\right)=0.
\end{equation}

Noting that the supercoercivity of $\varphi$ implies 
$\dom\varphi^*=\R$, it follows from Proposition~\ref{pro:plc_radial} 
that 
\begin{equation}
\label{eq:plc_radial_geq0_prox2}
P_{\epi 
\widetilde{f}}(x,\eta,\delta)= 
        \begin{cases}
(0,0,\delta+\mu),&\text{if}\:\:\eta+\mu\varphi^*(\|x\|/\mu)\le 0;\\[2mm]
\big(0,\eta+\mu\varphi^*(0),\delta+\mu\big), 
&\text{if}\:\:\eta+\mu\varphi^*(0)>0
\:\:\text{and}\:\:x= 0;\\[2mm]
\left(\prox_{\frac{\mu}{\nu}\varphi} \big(\frac{\|x\|}{\nu}\big)
\frac{\nu 
x}{\|x\|},\nu,\delta+\mu\right),&\text{if}\:\:\eta+\mu\varphi^*(\|x\|/\mu)>0
\:\:\text{and}\:\:x\neq 0,
        \end{cases}
    \end{equation}
where $\nu>0$ is the unique solution to 
\begin{equation}
\label{eq:plc_radial_geq0_mur}
\nu = \eta + \mu \varphi^* 
\left(\prox_{\frac{\nu}{\mu}\varphi^*}\left(\frac{\|x\|}{\mu}\right)\right).
\end{equation} 
We now divide the remainder of the proof into three parts, 
following \eqref{eq:plc_radial_geq0_prox2}:
\begin{enumerate}
\item $\underline{\eta+\mu\varphi^*(\|x\|/\mu)\le 
0\:\Leftrightarrow \:
[\eta-\delta\varphi^*(\|x\|/(-\delta))\le 0\:\: \text{and}\:\: \delta<0}]$:
Indeed, if $\eta+\mu\varphi^*(\|x\|/\mu)\le 0$, 
$\widetilde{f}(\prox_{\mu 
\widetilde{f}}(x,\eta))=\rec f(0)=0$, and we obtain from 
\eqref{e:eqmurad} that $\mu=-\delta>0$, and we get
$\eta-\delta\varphi^*(\|x\|/(-\delta))\le 0$.
Conversely, if $\eta-\delta\varphi^*(\|x\|/(-\delta))\le 0$ and 
$\delta<0$, it follows from Proposition~\ref{pro:plc_radial} that 
$ \prox_{-\delta\widetilde{f}}(x,\eta)=(0,0)$ and, therefore, 
Definition~\ref{def:perspective} yields
$\widetilde{f}(\prox_{-\delta\widetilde{f}}(x,\eta))=0$.
We conclude from \eqref{e:eqmurad} that $\mu=-\delta$ and the 
result follows.

\item $\underline{[x=0\:\:\text{and}\:\:\eta+\mu\varphi^*(0)>0]\:\:
\Leftrightarrow\:\: [x=0\:\:\text{and}\:\:\eta-\delta\varphi^*(0)>0]}:$
Since $\varphi$ is even,
\begin{equation}
\label{e:parity}
\varphi^*(0)=-\inf_{x\in\R}\varphi(x)=-\varphi(0),
\end{equation}  
which implies $1+(\varphi^*(0))^2\ge 1>0$. Now, if $x=0$ and 
$\eta+\mu\varphi^*(0)>0$,
we have from \eqref{eq:plc_radial_geq0_prox2} and 
Definition~\ref{def:perspective} that 
$\widetilde{f}(\prox_{\mu 
\widetilde{f}}(x,\eta))=(\eta+\mu\varphi^*(0))\varphi(0)$. Hence,
\eqref{e:eqmurad} reduces to 
\begin{equation}
\mu=\frac{-\eta\varphi^*(0)-\delta}{1+(\varphi^*(0))^2},
\end{equation}
which yields
\begin{equation}
\label{e:rightaux}
0<\eta+\mu\varphi^*(0)=
\frac{\eta-\delta\varphi^*(0)}{1+(\varphi^*(0))^2}.
\end{equation}
Conversely, if $x=0$ and $\eta-\delta\varphi^*(0)>0$, 
we have that 
\begin{equation}
\label{e:auxmut}
0<\frac{\eta-\delta\varphi^*(0)}{1+(\varphi^*(0))^2}
=\eta+\left(\frac{-\eta\varphi^*(0)-\delta}{1+(\varphi^*(0))^2}\right)
\varphi^*(0).
\end{equation}
Now set
$\widehat{\mu}=\frac{-\eta\varphi^*(0)-\delta}{1+(\varphi^*(0))^2}$ 
and let us prove that $\widehat{\mu}>0$. Indeed,
condition $\widetilde{\varphi}(\|x\|,\max\{\eta,0\})>\delta$ yields two 
cases: either $\eta>0$ and $\eta\varphi(0)>\delta$ or $\eta\le 0$ 
and $\delta<0$. In the first case, it is direct from \eqref{e:parity} 
that $-\eta\varphi^*(0)-\delta>0$ and, therefore, $\widehat{\mu}>0$.
In the second case, it follows from 
$\eta-\delta\varphi^*(0)>0$ that $\varphi^*(0)>0$, which yields 
$-\eta\varphi^*(0)-\delta>0$.

We deduce from \eqref{e:auxmut} that  
$\eta+\widehat{\mu}\varphi^*(0)>0$ and 
Proposition~\ref{pro:plc_radial} implies that 
$\prox_{\widehat{\mu}\widetilde{f}}(x,\eta)
=(0,\eta+\widehat{\mu}\varphi^*(0))$. Hence, 
Definition~\ref{def:perspective} implies 
$$\widetilde{f}(\prox_{\widehat{\mu}\widetilde{f}}(x,\eta))
=(\eta+\widehat{\mu}\varphi^*(0))\varphi(0)=
\frac{\varphi^*(0)(\delta\varphi^*(0)-\eta)}{1+(\varphi^*(0))^2}
=\delta+\widehat{\mu},$$
\eqref{e:eqmurad} yields $\mu=\widehat{\mu}$, and the result 
follows.

\item $\underline{\eta+\mu\varphi^*(\|x\|/\mu)>0
\:\:\text{and}\:\:x\neq 0}$: In this case, it follows from 
\eqref{eq:plc_radial_geq0_prox2} and
Definition~\ref{def:perspective}
that 
\begin{equation}
\widetilde{f}(\prox_{\mu 
\widetilde{f}}(x,\eta))=\nu\varphi\Big(\prox_{\frac{\mu}{\nu} 
            \varphi} \Big(\frac{\|x\|}{\nu}\Big)\Big),
\end{equation}
where $\nu>0$ is the unique solution to 
\eqref{eq:plc_radial_geq0_mur}, and we deduce from 
\eqref{e:eqmurad} that 
$\mu$ solves \eqref{e:nu_radial}.
\end{enumerate}
The proof is complete.
\end{proof}

In order to efficiently solve the nonlinear scalar equation in 
\eqref{e:eqmain} in the case when $\widetilde{f}
\big(P_{\, \cdom \widetilde{f}} (x,\eta)\big) > \delta$, define
\begin{equation}
\phi\colon\left]0,-\delta+ 
	\widetilde{f}(P_{\,\cdom \widetilde{f}}(x,\eta))\right]\to\RX\colon 
	\mu\mapsto 	\mu +\delta 
	-\widetilde{f}\left(\prox_{\mu\widetilde{f}}(x,\eta)\right).
\end{equation}
Note that, in view of Proposition~\ref{p:f_prox}, 
\begin{equation}
\phi\colon 
	\mu\mapsto \begin{cases}
\mu +\delta 
	-(\rec f)\big(\prox_{\mu(\rec f)}x\big),
&\text{if}\:\: \eta+{\mu}
f^*\Big(P_{\,\cdom f^*}\big(\frac{x}{{\mu}}\big)\Big)\le 0;\\
\mu +\delta 
	-\nu f\Big(\prox_{\frac{\mu}{\nu} 
f}\big(\frac{x}{\nu}\big)\Big),
&\text{if}\:\: \eta+{\mu}
f^*\Big(P_{\,\cdom f^*}\big(\frac{x}{{\mu}}\big)\Big)> 0,
\end{cases}
\end{equation}
where
$\nu\in\big]0,\eta + \mu f^*(P_{\cdom \, f^*} 
(x/\mu))\big]\cap\R$ is the unique solution to $\psi_{\mu}(\nu)=0$ 
and
\begin{equation}
(\forall \mu>0)\quad\psi_{\mu}:\big]0,\eta + \mu f^*(P_{\cdom \, f^*} 
(x/\mu))\big]\to\RX: \nu \mapsto \nu - \eta - \mu 
f^*\left(\prox_{\frac{\nu}{\mu} f^*}\left(\frac{x}{\mu}\right)\right).
\end{equation}
Then, it follows from 
\cite[Lemma~3.3(iii)]{briceno2018proximal} that 
$\phi(\mu)\to \delta-\widetilde{f}
\big(P_{\, \cdom \widetilde{f}} (x,\eta)\big)<0$ as $\mu\downarrow 
0$. Hence, we extend the domain of $\phi$ to $[0,-\delta+ 
	\widetilde{f}(P_{\,\cdom \widetilde{f}}(x,\eta))]$ by defining
\begin{equation}
\label{e:mubar}
\overline{\phi}\colon \mu\mapsto 
\begin{cases}
\delta-\widetilde{f}
\big(P_{\, \cdom \widetilde{f}} (x,\eta)\big),&\text{if}\:\: \mu=0;\\
\phi(\mu),&\text{if}\:\: \mu>0.
\end{cases}
\end{equation}
The following Algorithm~\ref{alg:P_epi} implements a bisection 
procedure to find a zero of the function $\bar{\phi}$ when 
$\widetilde{f}
\big(P_{\, \cdom \widetilde{f}} (x,\eta)\big) > \delta$. Of course, 
there exist
alternative one-dimensional root-finding algorithms able to perform 
this task.

\begin{algorithm}[H]
	\caption{Projection of $(x,\eta,\delta)\in\HH\times\R^2$ onto 
	$\epi \widetilde{f}$ when $\widetilde{f}(P_{\overline{\rm dom} 
	\widetilde{f}}(x,\eta)) >\delta$ with tolerance 
$\epsilon>0$.}
	\label{alg:P_epi}
	\KwData{$(x,\eta,\delta)\in\HH\times\R^2$%and $\epsilon>0$ 
	%representing the numerical zero (e.g., $\epsilon=10^{-15}$)
	.}
	\KwResult{$(\bar{x},\bar{\eta},\bar{\delta})$.}

	{ Set $\mu^0_-=0$ and $\mu^0_+=\begin{cases}
-\delta + \widetilde{f}
	\big(P_{\, \cdom \widetilde{f}} (x,\eta)\big),&\text{if}\:\:P_{\, 
	\cdom \widetilde{f}} (x,\eta) \in \dom \widetilde{f};\\
N^k,&\text{if}\:\:P_{\, \cdom \widetilde{f}} (x,\eta) \notin \dom 
\widetilde{f},
\end{cases}$\\
where $N>1$ and $k\in\N$ is the first integer satisfying 
$\phi(N^k)>0$ (see Remark~\ref{r:phi}\ref{r:phi_prop}).\\
\vspace{.5\baselineskip}
Set $m=[\log_2(\mu^0_+/ \epsilon)]$.

\For{$n=0$ \KwTo $m$}{
		{
			\vspace{.5\baselineskip}
			$\hat{\mu}^{n+1} = \frac{\mu_+^{n} +\mu_-^{n}}{2}$
			\vspace{.5\baselineskip}
		} \\
		\eIf{$\eta+\hat{\mu}^{n+1}
			f^*\Big(P_{\,\cdom 
			f^*}\big(\frac{x}{\hat{\mu}^{n+1}}\big)\Big)\le 0$ 		
}
{\vspace{.5\baselineskip}
	$(\bar{x},\bar{\eta}) =\big(\prox_{\hat{\mu}^{n+1}(\rec f)}x,0\big)$\\
	\vspace{.5\baselineskip}
\eIf{$\hat{\mu}^{n+1} +\delta 
			-(\rec f)(\bar{x})>0$}
			{
				$\mu_+^{n+1} = \hat{\mu}^{n+1}$
			}
			{
				$\mu_-^{n+1} = \hat{\mu}^{n+1}$
			}
}
{Set $\hat{\nu}^{n+1}$ as the solution to 
		$\psi_{\hat{\mu}^{n+1}}(\nu)=0$.\\
		\vspace{.5\baselineskip}
$(\bar{x},\bar{\eta}) 
=\Big(\hat{\nu}^{n+1}\prox_{\frac{\hat{\mu}^{n+1}}{\hat{\nu}^{n+1}} 
				f}\big(\frac{x}{\hat{\nu}^{n+1}}\big),\hat{\nu}^{n+1}\Big)$\\
				\vspace{.5\baselineskip}
		\eIf{$\eta + \hat{\mu}^{n+1} + \delta - \bar{\eta} 
		f(\bar{x}/\bar{\eta})> 0$}
			{
				$\mu_+^{n+1} = \hat{\mu}^{n+1}$
			}
			{
				$\mu_-^{n+1} = \hat{\mu}^{n+1}$
			}}

%		{
%			\vspace{.5\baselineskip}
%			$\mu_+^{n+1}=\hat{\mu}^{n+1}.$
%			\vspace{.5\baselineskip}\\
%		}
%		\uElseIf{$\eta+\hat{\mu}^{n+1}
%			f^*\Big(P_{\,\cdom 
%f^*}\big(\frac{x}{\hat{\mu}^{n+1}}\big)\Big)\le 0$ 
%			and $\hat{\mu}^{n+1} +\delta 
%			-(\rec f)\big(\prox_{\hat{\mu}^{n+1}(\rec f)}x\big)\le 0$}
%		{
%			\vspace{.5\baselineskip}
%			$\mu_-^{n+1}=\hat{\mu}^{n+1}.$
%			\vspace{.5\baselineskip}\\
%		}
%		\ElseIf{$\eta+\hat{\mu}^{n+1}
%			f^*\Big(P_{\,\cdom 
%f^*}\big(\frac{x}{\hat{\mu}^{n+1}}\big)\Big)> 0$}
%		{
%			\vspace{.5\baselineskip}
%		
%		}	

%	\eIf{$\eta+\hat{\mu}^{m+1}
%		f^*\Big(P_{\,\cdom 
%f^*}\big(\frac{x}{\hat{\mu}^{m+1}}\big)\Big)> 0$}
%		{
%		$\prox_{\hat{\mu}^{m+1} \widetilde{f}}(x,\eta) = 
%\Big(x-\hat{\mu}^{m+1} P_{\dom 
%f^*}\Big(\frac{x}{\hat{\mu}^{m+1}}\Big),0\Big)$
%		}
%	{
%		$\prox_{\hat{\mu}^{m+1} \widetilde{f}}(x,\eta) = 
%\Big(x-\hat{\mu}^{m+1} 
%\prox_{\frac{\hat{\mu}^{m+1}}{\hat{\nu}^{m}}f^*}\Big(\frac{x}{\hat{\nu}^{m+1}}\Big),\hat{\nu}^{m+1}\Big)$
%	}
}
	$\bar{\delta} =\delta + \hat{\mu}^{m+1}$
}
\end{algorithm}		

Next result provides explicit error bounds for 
Algorithm~\ref{alg:P_epi}.

\begin{theorem}
Let $(x,\eta,\delta)\in\HH\times\R^2$ be such that $ 
\widetilde{f}(P_{\,\cdom 
\widetilde{f}}(x,\eta)) >\delta$, let 
$(x^*,\eta^*,\delta^*)=P_{\epi\widetilde{f}}(x,\eta,\delta)$, and let 
$(\bar{x},\bar{\eta},\bar{\delta})$ 
be the vector obtained by
Algorithm~\ref{alg:P_epi}. Then
\begin{equation}
|\delta^*-\bar{\delta}|\le \epsilon\quad \text{and}\quad 
\|(x^*,\eta^*)-(\bar{x},\bar{\eta})\|\leq 
\frac{\epsilon}{\bar{\delta}-\delta}\cdot 
\|(x,\eta)-(\bar{x},\bar{\eta})\|.
\end{equation}
\end{theorem}
\begin{proof}
Note that, since $\widetilde{f}(P_{\,\cdom 
\widetilde{f}}(x,\eta)) >\delta$, Algorithm~\ref{alg:P_epi} yields
$\mu^{m+1}_+>0$ and, hence, 
$\bar{\phi}(\mu_-^{m+1})\le 0$ and 
$\bar{\phi}(\mu_+^{m+1}) > 0$, where $\bar{\phi}$ is defined in 
\eqref{e:mubar}. 
Moreover, let $\bar{\mu}>0$ be the unique solution to 
$\bar{\phi}(\mu)=0$. It follows from Theorem~\ref{t:main_result} 
that
$\delta^*=\delta+\bar{\mu}$
and, since \cite[Lemma 3.3(iii)]{bricenoarias2023proximity} asserts 
that $\bar{\phi}$ is continuous, 
$\bar{\mu}\in[\mu_-^{m+1},\mu_+^{m+1}]$.
%
%, and $\bar{\phi}$ is continuous in $[\mu_-^0, \mu_+^0]$ (see 
%Remark~\ref{r:phi}\ref{r:phi_prop}). Hence 
%the bisection method converges \cv{to the unique solution of 
%$\bar{\phi} = 0$, unnecessary?}.

On the other hand, from Algorithm~\ref{alg:P_epi} we get
\begin{equation}
	\label{e:hat}
	(\bar{x},\bar{\eta}) = \prox_{\hat{\mu}^{m+1} 
	\widetilde{f}}(x,\eta)\quad\text{and}\quad 
	\bar{\delta}=\delta + \hat{\mu}^{m+1},
\end{equation}
and, by construction of $\hat{\mu}^{m+1}$,
\begin{equation}
	\label{e:delta_Hat}
	\abs{\delta^* - \bar{\delta}} = \abs{\bar{\mu} - \hat{\mu}^{m+1}} 
	\le \frac{\mu_+^0 - \mu_-^0}{2^{m+1}} \le
	\frac{\mu_+^0}{2^{\log_2(\mu_+^0/\epsilon)}} = 
	\frac{\mu_+^0}{\mu_+^0/\epsilon} = \epsilon,
\end{equation}
which yields the first inequality.
On the other hand, since 
Lemma~\ref{l:persp_prop}\ref{l:persp_prop_i} implies 
$\widetilde{f} \in \Gamma_0(\HH \oplus \R)$, we deduce from 
\eqref{e:hat}, Theorem~\ref{t:main_result},
\cite[Proposition~23.31]{bauschke2011convex}, and 
\eqref{e:delta_Hat} that 
\begin{align}
	\label{e:xeta_hat}
	\norm{(\bar{x},\bar{\eta})-(x^*,\eta^*)} &= 
	\norm{\prox_{\hat{\mu}^{m+1} \widetilde{f}}(x,\eta) - 
	\prox_{\bar{\mu} \widetilde{f}}(x,\eta)} \nonumber\\
&\le \abs{1- 
	\frac{\bar{\mu}}{\hat{\mu}^{m+1}}} \norm{ 
	(x,\eta)-\prox_{\hat{\mu}^{m+1} 
	\widetilde{f}}(x,\eta)} 
\nonumber\\
%	& = \abs{\frac{\hat{\mu}^{m+1} - \bar{\mu}}{\hat{\mu}^{m+1}}} 
%\norm{(x,\eta) - \prox_{\hat{\mu}^{m+1} \widetilde{f}}(x,\eta)} 
%\nonumber\\
%	& 
&= \frac{\epsilon}{\hat{\mu}^{m+1}}\norm{(x,\eta) - 
(\bar{x},\bar{\eta})},
\end{align}
and the result follows from \eqref{e:hat}. 
\end{proof}

Now we provide computations of the projection onto the 
epigraphs of the perspectives of two examples of functions 
$f\in\Gamma_{0}(\R)$. This computations will motivate our 
numerical results in 
Section~\ref{s:4}.

\begin{example}[Exponential cone]
	\label{ex:exp_cone}
Suppose that $\HH=\R$ and $f\colon \xi\mapsto e^\xi$. Then
$\epi{\widetilde{f}}=\mathcal{K}_{\exp}$, which is the exponential 
cone
(see, e.g., \cite[Section~4.2]{Chares2009},
\cite{friberg2023}, and \cite{mosek2022paper,Chares2009} for 
applications). 

Fix $(x,\eta,\delta) \in \R^3$.
In order to compute $P_{\epi \widetilde{f}}(x,\eta,\delta)$, note that 
$f \in \Gamma_0(\R)$ and using \eqref{e:recession} we obtain 
$\rec f = 
\iota_{]-\infty,0]}$, which yields
\begin{equation}
	\label{e:exp_persp}
	\widetilde{f}: \R \times \R \to \RX: (x,\eta) \mapsto \begin{cases}
		\eta e^{\frac{x}{\eta}}, & \text{if}\:\: \eta > 0;\\
		\iota_{]-\infty,0]}(x), & \text{if}\:\: \eta = 0;\\
		+\infty, & \text{if}\:\: \eta < 0.\\
	\end{cases}
\end{equation}
Hence, $\cdom \widetilde{f} = \R \times 
\left[0,+\infty\right[$, $P_{\cdom \widetilde{f}}: (x,\eta) \mapsto (x, 
\max\{0,\eta\})$, and 
\begin{equation}
	\label{e:fpcdom}
	\widetilde{f}(P_{\cdom \widetilde{f}}(x,\eta)) = \begin{cases}
		\eta e^{\frac{x}{\eta}}, & \text{if}\:\: \eta > 0 ; \\
		0, & \text{if}\:\: \eta \le 0\:\:\text{and}\:\:x\le 0; \\
+\infty, & \text{if}\:\: \eta \le 0\:\:\text{and}\:\:x> 0.
	\end{cases} 
\end{equation}
Now, note that
\begin{equation}
	\label{e:exp_cond}
	\widetilde{f} \left(P_{\cdom \widetilde{f}}(x,\eta)\right) \le \delta 
	\quad \Leftrightarrow\quad  \left(\eta > 0 \text{ and }\eta 
	e^{\frac{x}{\eta}} \le \delta\right) \text{ or } \left(\eta \le 0,\:x \le 0, 
	\text{ and } 0 \le \delta\right).
\end{equation}
Altogether, Theorem~\ref{t:main_result} yields
\begin{equation}
	\label{e:e_proj}
P_{\epi \widetilde{f}}(x,\eta,\delta)=
	\begin{cases}
		\big(x,0,\delta\big),&\text{if}\:\: 
		\eta \le 0,\:\:x \le 0, \text{ and } 0 \le \delta;\\[2mm]
		\big(x,\eta,\delta\big),&\text{if}\:\: 
		\eta > 0 \text{ and }\eta e^{\frac{x}{\eta}}\le \delta;\\[2mm]
		\big(\prox_{\mu\widetilde{f}}(x,\eta),\delta + 
		\mu\big),&\text{otherwise},
	\end{cases}
\end{equation}
where $\mu\in\left]0,-\delta+\max\{0,\eta e^{x/\eta}\}\right]$
is the unique solution to
\begin{equation}
	\label{e:mu_e_equal}
	\mu + \delta - \widetilde{f}\big(\prox_{\mu 
	\widetilde{f}}(x,\eta)\big) = 0.
\end{equation}
Next, in order to compute $\prox_{\mu \widetilde{f}}(x,\eta)$ and 
$\widetilde{f}(\prox_{\mu \widetilde{f}}(x,\eta))$ we use 
 Proposition~\ref{p:f_prox}. Since 
\cite[Example~13.2(v)]{bauschke2011convex} asserts that
\begin{equation}
	\label{e:expstar}
	f^*: \R \to \RX: \xi \mapsto \begin{cases}
		\xi \left(\ln \xi -1\right), & \text{if}\:\: \xi > 0;\\
		0, & \text{if}\:\: \xi = 0;\\
		+\infty, & \text{if}\:\: \xi < 0,
	\end{cases}
\end{equation}
we have $\cdom f^* = \RP$, $P_{\cdom f^*}: \xi \mapsto 
\max\{0,\xi\}$, and
\begin{align}
	\label{e:fproxf_cond}
	\eta + \mu f^*\left(P_{\cdom f^*} \left(\frac{x}{\mu}\right)\right) \le 
	0\quad &\Leftrightarrow\quad  (x \le 0 \text{ and } \eta \le 0) 
	\text{ or } 
	\left(x 
	> 0 \text{ and }\eta + x \left(\ln 
	\left(\frac{x}{\mu}\right) -1\right) \le 0\right) \nonumber\\
	& \Leftrightarrow\quad (x \le 0 \text{ and } \eta \le 0) \text{ or } 
	\left(x > 
	0 \text{ and } \mu  \ge x e^{\frac{\eta}{x}-1}\right).
\end{align}
Moreover, it follows from 
\cite[Example~24.39]{bauschke2011convex} that 
\begin{equation}
\label{e:prox_e}
(\forall \gamma \in \RPP)\quad \prox_{\gamma f}(x)=x-W(\gamma 
	e^x)\quad\text{and}\quad 
\prox_{\gamma f^*} (x) = \gamma 
W\left(\frac{e^{x/\gamma}}{\gamma} \right),
\end{equation}
where $W \colon [-1/e, +\infty [ \to [-1, +\infty[$ is the principal 
branch of the Lambert W-function, defined by the inverse of the 
function $\xi \mapsto \xi e^\xi$ on $[-1, +\infty[$ \cite{Corl96}.
In addition, since $\ln\circ W=\ln-W$ and $\prox_{\gamma 
f^*}(x)\in\dom\partial f^*=\RPP$, we obtain from \eqref{e:expstar} 
that
\begin{equation}
(\forall \gamma \in \RPP)\quad f(\prox_{\gamma 
f}(x))=\frac{1}{\gamma}W(\gamma 
e^x)\quad\text{and}\quad  f^*(\prox_{\gamma 
f^*}(x))=(x-\gamma) W\left(\frac{e^{x/\gamma}}{\gamma} 
\right)-\gamma 
\left(W\left(\frac{e^{x/\gamma}}{\gamma} \right)\right)^2.
\end{equation}
Hence, since $\rec f  = \iota_{]-\infty,0]}$, it follows from 
Proposition~\ref{p:f_prox} and \eqref{e:fproxf_cond} that
\begin{equation}
	\label{e:prox_persp_e}
	\prox_{\mu \widetilde{f}}(x,\eta)=\begin{cases}
(0,0),&\text{if}\:\: x > 0 \text{ and } \mu  \ge x 
		e^{\frac{\eta}{x}-1};\\[1mm]
		(x,0), &\text{if}\:\: x \le 0 \text{ and } \eta \le 
		0;\\[1mm]
		\left(x-\nu 
		W\left(\dfrac{\mu}{\nu}e^{\frac{x}{\nu}}\right),\nu\right),
&\text{otherwise,}
	\end{cases}	
\end{equation}
where $\nu>0$ 
is the unique solution to 
\begin{equation}
	\label{e:nu_equal}
	\nu =  \eta + (x-\nu)W\left(\frac{\mu}{\nu} e^{\frac{x}{\nu}}\right) - 
	\nu \left(W\left(\frac{\mu}{\nu} e^{\frac{x}{\nu}}\right)\right)^2,
\end{equation}
which is in $\left]0, \eta + x 
\left(\ln(x/\mu)-1\right)\right]$ if $x>0$ and in 
$\left]0,\eta \right]$ if $x\le 0$.
Furthermore, Proposition~\ref{p:f_prox} and \eqref{e:exp_persp} 
yield
\begin{equation}
	\label{e:fproxf_e}
\widetilde{f}\big(\prox_{\mu \widetilde{f}}(x,\eta)\big)
	=
	\begin{cases}
		0,&\text{if}\:\:
		(x \le 0 \text{ and } \eta \le 0) \text{ or } \left(x > 0 \text{ and } 
		\mu  \ge x e^{\frac{\eta}{x}-1}\right);\\[3mm]
		\frac{\nu^2}{\mu} 
		W\left(\frac{\mu}{\nu}e^{x/\nu}\right),&\text{otherwise},
	\end{cases}
\end{equation}
from which we can solve \eqref{e:mu_e_equal} and get an explicit 
computation for \eqref{e:e_proj}.

\end{example}

\begin{example}[Hyperbolic penalty]
	\label{ex:MBF}
Let $\HH=\R$ and
\begin{equation}
	f:\R\to \RX: \xi \to \begin{cases}
		\dfrac{\xi}{1-\xi}, & \text{if}\:\: \xi < 1; \\
		+\infty, & \text{otherwise.}
	\end{cases}
\end{equation}
This function appears as a hyperbolic penalization for 
constrained optimization problems 
\cite{auslender1997asymptotic} appearing, for instance, in the 
modified Carroll Function 
\cite{carroll1961created,polyak1992modified}. 

Fix $(x,\eta,\delta) \in \R^3$.
In order to compute $P_{\epi \widetilde{f}}(x,\eta,\delta)$, 
note that $f \in \Gamma_0(\R)$ and, using \eqref{e:recession}, we 
obtain $\rec f =\iota_{]-\infty,0]}$. Hence,
 it follows from \eqref{e:persp_def} that
\begin{equation}
	\label{e:mbf_persp}
	\widetilde{f}: \R \times \R \to \RX: (x,\eta) \mapsto \begin{cases}
		\dfrac{\eta x}{\eta - x}, & \text{if }\eta > 0\text{ and } x < 
		\eta;\\[2mm]
		\iota_{]-\infty,0]}(x), & \text{if }\eta = 0;\\[2mm]
		+\infty,&\text{otherwise},
	\end{cases}
\end{equation}
which yields $\cdom \widetilde{f} = \menge{(x,\eta)\in \R^2}{\eta \ge 0 
\text{ 
and } x\le \eta}$,
	\begin{equation}
		\label{e:mbf_pcdom}
P_{\cdom \widetilde{f}}(x,\eta) = 
		\begin{cases}
			(\min\{0,x\}, 0), & \text{if } \eta \le
			0 \text{ and } x \le -\eta;\\[2mm]
			\left(\dfrac{x+\eta}{2}, \dfrac{x+\eta}{2}\right)\!, & \text{if 
			}\abs{\eta} \le x;\\
			(x,\eta), &\text{otherwise},
		\end{cases}
	\end{equation} 
and
\begin{equation}
	\label{e:mbf_fpcdom}
 \widetilde{f}(P_{\cdom 
	\widetilde{f}}(x,\eta)) = \begin{cases}
		0,& \text{if } \eta \le
			0 \text{ and } x \le -\eta;\\[1mm]
		\dfrac{\eta x}{\eta-x}, &\text{if } \eta \ge 
		0\text{ and }  x < \eta;\\
		+\infty, & \text{otherwise}. 
	\end{cases}
\end{equation}
Hence, since
\begin{equation}
	\label{e:mbf_cond}
	\widetilde{f}\left(P_{\cdom \widetilde{f}} (x,\eta)\right)\le \delta \; 
	\ifaf \; \left(\eta \ge 0,\:\: x < \eta, \text{ and } \dfrac{\eta 
	x}{\eta - x} \le \delta\right) \text{ or } \left(\eta \le 0,\:\: x 
	\le -\eta, \text{ and } 0 \le  \delta\right),
\end{equation}
Theorem~\ref{t:main_result} yields
\begin{equation}
	\label{e:mbf_proj}
	P_{\epi\widetilde{f}}(x,\eta,\delta) = \begin{cases}
		(\min\{0,x\},0,\delta), &\text{if } \eta \le 0, \:\: x\le - 
		\eta, 
		\text{ 	and } 0 \le \delta; \\[2mm]
		(x,\eta,\delta) &\text{if } \eta \ge 0,\:\: x < \eta, \text{ 
		and } \dfrac{\eta x}{\eta-x} \le \delta; \\[2mm]
		(\prox_{\mu \widetilde{f}}(x,\eta),\delta + \mu), & 
		\text{otherwise},
	\end{cases}
\end{equation}
where $\mu$ 
is the unique strictly positive solution to
\begin{equation}
	\label{e:mbf_mu_equal}
	\mu + \delta - 
	\widetilde{f}\left(\prox_{\mu\widetilde{f}}(x,\eta)\right) = 0,
\end{equation}
which is in $]0,-\delta + \frac{\eta x}{\eta-x}]$ if 
$\eta 
		\ge 0 \text{ and } x < \eta$ and in $\left]0,-\delta\right]$
if $\eta < 0 \text{ and } x \le - \eta$.
Now, in order to compute $\prox_{\mu \widetilde{f}}(x,\eta)$ and 
$\widetilde{f}(\prox_{\mu \widetilde{f}}(x,\eta))$ we use 
 Proposition~\ref{p:f_prox}. Since 
\cite[Example~13.2(ii)\,\&\,Proposition~13.23(iii)\&(v)]{bauschke2011convex}
 imply that
\begin{equation}
	\label{e:mbf_star}
	f^*\colon \xi \mapsto \begin{cases}
		(\sqrt{\xi} - 1)^2, & \text{if}\:\: \xi \ge 0;\\
		+\infty, & \text{if}\:\: \xi < 0,
	\end{cases}
\end{equation}
we have $\cdom f^* = [0,+\infty[$, $P_{\cdom f^*}: \xi \mapsto 
\max\{0,\xi\}$, and
\begin{equation}
	\label{e:fstarP_mbf}
	(\forall \xi \in \R)\quad f^*(P_{\cdom f^*} \xi) = \begin{cases}
		(\sqrt{\xi}-1)^2, & \text{if}\:\: \xi \ge 0;\\
		1, & \text{if}\:\: \xi < 0.
	\end{cases}
\end{equation}

On the other hand, for every 
$\gamma \in \RPP$, it follows from \eqref{e:prox_char} that 
$\prox_{\gamma f}(x)$ is the unique solution in $]-\infty,1[$ of the 
cubic equation
\begin{equation}
	\label{e:mbf_prox}
(x-p)(1-p)^2 -\gamma = 0
\end{equation}
and from \eqref{e:moreau_decomp} we deduce that
$\prox_{\gamma f^*}(x)$ is the unique solution in 
$]\max\{x-\gamma,0\},+\infty[$ of the cubic 
equation 
\begin{align}
q^3 + 2 (\gamma - x)q^2 + q(\gamma-x)^2  - 
	\gamma^2 =0.
\end{align}
Hence, by noting that
\begin{equation}
	\label{e:mbf_fproxf_cond}
	\eta + \mu f^*\left(P_{\cdom f^*}\left(\dfrac{x}{\mu}\right)\right) 
	\le 0 \; \ifaf \; \eta + 
	\left(\sqrt{\max\{0,x\}}-\sqrt{\mu}\right)^2 \le 0,
\end{equation}
 it 
follows from Proposition~\ref{l:prox_persp} that
\begin{align}
	\label{e:mbf_prox_persp}
	\prox_{\mu \widetilde{f}}(x,\eta) 
	& = \begin{cases}
		\left(\min\{0,x\},0\right), & \text{if}
		\:\: \eta + 
	\left(\sqrt{\max\{0,x\}}-\sqrt{\mu}\right)^2 \le 0; 
		\\[3mm]
		\left(x-\mu \prox_{\frac{\nu}{\mu} 
		f^*}\left(\dfrac{x}{\mu}\right),\nu\right),& \text{otherwise},
	\end{cases}
\end{align}
where $\nu\in\:]0,\eta + 
	(\sqrt{\max\{0,x\}}-\sqrt{\mu})^2]$
is the unique solution to 
\begin{equation}
	\label{e:mbf_nu_equal}
	\nu = \eta + \mu f^*\left(\prox_{\frac{\nu}{\mu}f^*} 
	\left(\frac{x}{\mu}\right)\right) = \eta + \mu 
	\left(\sqrt{\prox_{\frac{\nu}{\mu}f^*}\left(\frac{x}{\mu}\right)}
-1\right)^2.
\end{equation}
Furthermore, by recalling that $\rec f = \iota_{]-\infty,0]}$, 
Proposition~\ref{p:f_prox} yields
\begin{align}
	\label{e:mbf_fproxf}
	\widetilde{f}(\prox_{\mu \widetilde{f}}(x,\eta)) & = \begin{cases}
		0,&\text{if}\:\: \eta + 
		\left(\sqrt{\max\{0,x\}}-\sqrt{\mu}\right)^2 \le 0;\\
		\nu \left(\dfrac{\prox_{\frac{\mu}{\nu} f} 
		\bigg(\dfrac{x}{\nu}\bigg)}{1-\prox_{\frac{\mu}{\nu} f} 
		\bigg(\dfrac{x}{\nu}\bigg)}\right), &\text{otherwise},
	\end{cases}
\end{align}
from which we can solve \eqref{e:mbf_mu_equal} and get an 
explicit 
computation for \eqref{e:mbf_proj}.

\end{example}
\section{Numerical experiments}
\label{s:4}
In this section we provide three numerical tests for the projection 
onto the cone generated by the epigraph of the perspective of 
some 
convex functions. In particular, we consider the exponential and 
hyperbolic cones, whose projections are computed in 
Example~\ref{ex:exp_cone} and Example~\ref{ex:MBF}, 
respectively. In each test, we generate random data using 
Lemma~\ref{l:normal}. 
In every experiment we use python 3.8.16 on 
an Intel i5 CPU at 1.60 GHz and 8GB of RAM.

\begin{test}
	\label{test:exp_cone}
	In the context of Example~\ref{ex:exp_cone}, using \eqref{e:exp_persp} we deduce that
	\begin{align*}
		\epi \widetilde{f}= \menge{(x,\eta,\delta)\in\R\times \RPP 
		\times \R}{\eta e^{\frac{x}{\eta}}\le \delta} \cup ~ ]-\infty,0] 
		\times \{0\} \times \RP,
	\end{align*}
which is the standard exponential cone $\mathcal{K}_{\exp}$ (see, 
e.g., 
\cite[Section~4.2]{Chares2009} and 
\cite{friberg2023}). Note that closed form expressions for 
$P_{\epi\widetilde{f}}$ are 
	available for points in  $\RM 
	\times \RM \times \R$ \cite[Theorem~3.1]{friberg2023}. Then,
	we consider 
	\begin{align}
		\mathcal{R}_1 &\coloneqq \menge{(x,\eta,\delta) \in \R \times 
		\RPP \times \R }{\varepsilon \le \eta \le 20,\:\: 0 \le x \le M 
		\cdot \eta, \:\:\text{and}\:\: \delta = \eta e^{\frac{x}{\eta}}}, \\
		\mathcal{R}_2 &\coloneqq \menge{(x,\eta,\delta) \in \R \times 
		\RPP \times \R }{\varepsilon \le \eta \le 20,\:\: -M \le x \le 
		0,\:\:\text{and}\:\: \delta = \eta e^{\frac{x}{\eta}}},
	\end{align}
which are in the boundary of $\epi \widetilde{f}$.

	These sets were chosen in order to avoid computational issues 
	with very large values of the exponential.
	Next, we set $\varepsilon = 10^{-15}$, $M = 
	10$, and we randomly generate $\{(\widehat{x}^1_i,\widehat{\eta}^1_i, 
	\widehat{\delta}_i^1)\}_{i = 
	1}^N \subset \mathcal{R}_1$ and $\{(\widehat{x}^2_i,\widehat{\eta}^2_i, 
	\widehat{\delta}^2_i)\}_{i = 1}^N 
	\subset 
	\mathcal{R}_2$, with $N =10000$ using the 
	\textit{random.uniform} function of python. Next, for every $i \in 
	\{1,\ldots,N\}$ and $j\in\{1,2\}$, we randomly choose $t^j_i \in 
	\left]0,10\right]$ and we set
	\begin{equation}
		(x_i^j,\eta_i^j, \delta_i^j) = 
		(\widehat{x}^j_i,\widehat{\eta}^j_i, \widehat{\delta}^j_i) + t^j_i \left(e^{\frac{\widehat{x}^j_i}{\widehat{\eta}^j_i}}, 
		\frac{e^{\frac{\widehat{x}^j_i}{\widehat{\eta}^j_i}}(\widehat{\eta}^j_i - \widehat{x}^j_i)}{\widehat{\eta}^j_i},-1\right).
	\end{equation}
	Noting that $\widetilde{f}$ is differentiable in $\R \times \RPP$, 
 we deduce from Lemma~\ref{l:normal} that,
	for every $i \in \{1,\ldots,N\}$ and $j\in\{1,2\}$, 
	$(x_i^j,\eta_i^j, 
	\delta_i^j)\notin \epi \widetilde{f}$ and  $(\widehat{x}^j_i,\widehat{\eta}^j_i, 
	\widehat{\delta}^j_i) 
	= P_{\epi \widetilde{f}}(x_i^j,\eta_i^j, 
	\delta_i^j) \coloneqq \widehat{p_i}^j$.
	
	Next, we approximate 
	$\{\widehat{p_i}^1,\widehat{p_i}^2\}_{i=1}^N$ using our 
	approach and the 
	available free source software SCS \cite{scs2016}. For our 
	approach, for every $i \in 
	\{1,\ldots,N\}$ and $j\in\{1,2\}$, we compute the unique solution
$\mu_i^j\in\RPP$ to	\eqref{e:mu_e_equal} with $(x_i^j,\eta_i^j, 
	\delta_i^j)$ by using \textit{scipy.optimize.} 
	\textit{root\_scalar} \cite{scipy} 
	with Brent's method \cite[Section~9.3]{press2007numerical} 
	with tolerance $\epsilon_1=10^{-9}$. Moreover, following 
	\eqref{e:e_proj}, we 
	define 
	our approximated projection
	\begin{align}
		\label{e:e_proj_approx1}
		p_i^j = 
		\begin{cases}
			\big(x_i^j,0,\delta_i^j\big),&\text{if}\:\: 
			\eta_i^j \le 0,\:\:x_i^j \le 0,
			\text{ and } 0 \le \delta_i^j;\\[2mm]
			\big(x_i^j,\eta_i^j,
\delta_i^j\big),&\text{if}\:\: 
			\eta_i^j > 0 \text{ and }\eta_i^j 
			e^{\frac{x_i^j}{\eta_i^j}}\le 
			\delta_i^j;\\[2mm]
			\Big(\prox_{\mu_i^j\widetilde{f}}(x_i^j,
\eta_i^j),\delta_i^j + 
			\mu_i^j\Big),&\text{otherwise}.
		\end{cases}
	\end{align}

	On the other hand, for every $i \in \{1,\ldots,N\}$ and 
	$j\in\{1,2\}$, we denote
%		\item $p_i^{f1}$ and $p_i^{f2}$ to approximations for 
%$\widetilde{p_i}^1$ and $\widetilde{p_i}^2$ using the code in 
%\url{http://github.com/HFriberg/projection} developed in 
%\cite{friberg2023}
%		\item $p_i^{c1}$ and $p_i^{c2}$ to approximations for 
%$\widetilde{p_i}^1$ and $\widetilde{p_i}^2$ using the solver 
%COSMO \cite{cosmo2021}.
by $p_i^{sj}$ the approximation of $\widehat{p_i}^j$ using the 
solver SCS  with the same tolerance $\epsilon_1$ over duality 
gap, primal, and dual residuals \cite{scs2016}.
%This choice is the minimal tolerance allowed by the method 
%numerically

In Table~\ref{tab:R1} we exhibit the average and standard 
deviation of the errors $\{\|p_i^j - 
\widehat{p_i}^j\|_2\}_{i=1}^N$,
$\{\|p_i^{sj} - \widehat{p_i}^j\|_2\}_{i=1}^N$ for $j \in \{1,2\}$, and provide the average of the 
computational time used for each approximation. 
We consider the standard 2-norm $\|\cdot\|_2$ in $\R^3$.
We observe that our method is less precise and is a bit slower 
than SCS in region $\mathcal{R}_1$, while we observe a 
significant improvement in precision and computational time 
 with respect to SCS in region $\mathcal{R}_2$.
We attribute this difference in the numerical behavior to 
inefficiences on the resolution of the scalar equations in 
our approach when the exponential achieves very high 
values for points generated from region $\mathcal{R}_1$.

	\begin{table}[ht]
		\centering
		\begin{tabular}{l cc cc}
			\toprule
			 & \multicolumn{2}{c}{$\mathcal{R}_1$} & \multicolumn{2}{c}{$\mathcal{R}_2$}\\
			 \cmidrule[\heavyrulewidth](lr){2-3} \cmidrule[\heavyrulewidth](lr){4-5}
			Approach &  \eqref{e:e_proj_approx1} & SCS 
			\cite{scs2016} & \eqref{e:e_proj_approx1} & SCS 
			\cite{scs2016} \\
			\midrule[\heavyrulewidth]
			Error Average       & 3.12e-03  & 2.22e-05 & 8.85e-14 &  4.69e-08 \\
			Error St. Deviation & 9.25e-03 & 6.21e-05 & 1.50e-13 &  1.42e-06 \\
			Av. comput. time [ms] &  100.70  &   60.20 & 21.35 &     77.80 \\
			\bottomrule
		\end{tabular}
		\caption{Average and standard deviation for the errors and 
		average time (in milliseconds) in the 
		computation of projections 
		on $\mathcal{R}_1$ and $\mathcal{R}_2$, using our 
		approach and 
		SCS \cite{scs2016}, when $N$ = 10000.}
		\label{tab:R1}
	\end{table}

\end{test}

In order to replicate the numerical comparison on 
Test~\ref{test:exp_cone} in higher dimension, the next numerical 
test is a radial version of the former using 
Proposition~\ref{p:fproxf_radial}.

\begin{test}
In this test we use Proposition~\ref{p:fproxf_radial} with 
$\varphi=e^{|\cdot|}$ and $\HH = \R^{n}$, where $n = 10000$. In 
addition, we deduce from 
\eqref{e:persp_def} and \eqref{e:recession} that
\begin{equation}
	\widetilde{f}:\R^{n} \times \R\to\RX\colon (x,\eta) \mapsto 
	\begin{cases}
		\eta e^{\frac{\|x\|}{\eta}}, & \text{if} \:\: \eta > 0;\\
		0, & \text{if}\:\: \eta = 0 \text{ and } x = 0; \\
		+\infty, & \text{if}\:\: \eta < 0,
	\end{cases}
\end{equation}
and
\begin{equation}
	\epi\widetilde{f} = \menge{(x,\eta,\delta) \in \R^{n} \times \RPP 
	\times \R}{\eta e^{\frac{\|x\|}{\eta}} \le \delta} \cup \{(0,0)\} \times 
	\RP.
\end{equation}
Note that this cone is related to the exponential cone via 
$$(x,\eta,\delta) \in \epi{\widetilde{f}}\quad\text{if and only if}\quad  
(\|x\|,\eta,\delta)\in\mathcal{K}_{\exp}.$$

In order to generate our synthetic data,
we consider the following subset of the 
boundary of $\epi{\widetilde{f}}$,
\begin{equation}
	\mathcal{R}_3 \coloneqq \menge{(x,\eta,\delta) \in \R^n \times 
	\RP \times \R}{\varepsilon \le \eta \le 10, \:\:0 < \norm{x} \le M 
	\cdot \eta, \text{ and } \delta = \eta e^{\frac{\norm{x}}{\eta}}}.
\end{equation}
Next, 
we set $N =1000$, $\varepsilon = 1$, and $M = 5$ in 
$\mathcal{R}_3$, in order to avoid numerical issues with very 
large values for the exponential function. For every $i\in\{1,\ldots,N\}$, we 
randomly generate
$(\widehat{x}_i,\widehat{\eta}_i,\widehat{\delta}_i)\in \mathcal{R}_3$ and $t_i 
\in\left]0,1\right[$, and we set
\begin{equation}
	(x_i,\eta_i, \delta_i) = (\widehat{x}_i,\widehat{\eta}_i, \widehat{\delta}_i) + t_i \left(\frac{\widehat{x}_i}{\norm{\widehat{x}_i}} e^{\frac{\norm{\widehat{x}_i}}{\widehat{\eta}_i}}, e^{\frac{\norm{\widehat{x}_i}}{\widehat{\eta}_i}} \left(1-\frac{\norm{\widehat{x}_i}}{\widehat{\eta}_i}\right),-1\right).
\end{equation}
Since $\widetilde{f}$ is differentiable in $\R^n\setminus\{0\}
\times \RPP$, Lemma~\ref{l:normal} implies that $(x_i,\eta_i, \delta_i) \notin \epi \widetilde{f}$ and $(\widehat{x}_i,\widehat{\eta}_i, 
\widehat{\delta}_i) = P_{\epi \widetilde{f}}(x_i,\eta_i, 
\delta_i) \coloneqq \widehat{p_i}$.
Next, in view of Proposition~\ref{p:fproxf_radial}, for every $i \in 
\{1,\ldots,N\}$, we approximate $\widehat{p_i}$ by setting
\begin{equation}
	\label{e:p_i_radial}
	p_i = 
		\begin{cases}
(x_i,\max\{\eta_i,0\},\delta_i),&\text{if}\:\:
\widetilde{\varphi}(\|x_i\|,\max\{\eta_i,0\})\le \delta_i;\\[1mm]
(0,0,0),&\text{if}\:\:
\widetilde{\varphi}(\|x_i\|,\max\{\eta_i,0\})> \delta_i,\:\:
\delta_i<0,\:\:\text{and}\:\:\eta_i-\delta_i\varphi^*(\frac{\|x_i\|}{-\delta_i})\le 
0;\\[1mm]
\left(0,\frac{\eta_i-\delta_i\varphi^*(0)}{1+(\varphi^*(0))^2},
\frac{\varphi^*(0)(\delta_i\varphi^*(0)-\eta_i)}{1+(\varphi^*(0))^2}\right),
&\text{if}\:\:\widetilde{\varphi}(\|x_i\|,\max\{\eta_i,0\})> 
\delta_i,\:\:\:\:\eta_i-\delta_i\varphi^*(0)>0,
\:\:\text{and}\:\:x_i= 0;\\[2mm]
\left(\prox_{\frac{\mu_i}{\nu_i}\varphi}
\left(\frac{\|x_i\|}{\nu_i}\right)\frac{\nu_i 
	x_i}{\|x_i\|},\nu_i,\delta_i+\mu_i\right),&\text{otherwise},
	\end{cases}
\end{equation}
where $(\mu_i,\nu_i) \in \RPP^2$ is the approximate solution of 
\eqref{e:mu_radial}-\eqref{e:nu_radial} using mainly Nelder-Mead 
algorithm \cite{NM} in the library 
\textit{scipy.optimize.minimize} \cite{scipy} for minimizing the 
quadratic residual of the scalar equations. We consider 
$\epsilon_2=5\cdot10^{-10}$ as tolerance for the stopping criterion 
in the 
resolution of the scalar system. On the other hand, for 
every $i \in \{1,\ldots,N\}$, we denote by $p_i^{s}$ to the 
approximation of $\widehat{p_i}$ using the solver SCS 
\cite{scs2016} with the same tolerance $\epsilon_2$ over duality 
gap, primal, and dual residuals.

In Table~\ref{tab:exp_norm} we exhibit the average and standard 
deviation of the errors $\left\{\norm{p_i - 
\widehat{p_i}}_2\right\}_{i=1}^N$, together with the average 
computational time to achieve the tolerane $\epsilon_2$. We 
observe that our method is several orders of magnitude more 
precise than 
SCS with a similar 
computational time, in average. 
%Moreover, the errors and times 
%are similar to 
%the one dimensional case studied in Test~\ref{test:exp_cone}.
	\begin{table}[H]
		\centering
		\begin{tabular}{lcc}
			Approach &  \eqref{e:p_i_radial} & SCS\\
			\toprule
			Error Average       & 9.55e-14 & 2.90e-09 \\
			Error St. Deviation & 2.23e-13 & 8.47e-09\\
			Av. comput. time [ms] &  29.97 & 28.03\\
			\bottomrule
		\end{tabular}
		\caption{
Average and standard deviation for the errors and 
		average time (in milliseconds) in the 
		computation of projections 
		on $\mathcal{R}_3$, using our 
		approach and 
		SCS \cite{scs2016}, when $N = 1000$.}
		\label{tab:exp_norm}
	\end{table}
	
%	\begin{table}[H]
%		\label{tab:exp_norm_system}
%		\centering
%		\begin{tabular}{lcc}
%			Approach &  \eqref{e:p_i_radial} & SCS\\
%			\toprule
%			Error Average       & 1.37e-13 & 6.05e-09 \\
%			Error St. Deviation & 1.18e-12 & 2.48e-09\\
%			Av. comput. time [ms] &  147.70 & 29.32\\
%			Tolerance & 5e-10 & \\ 
%			\bottomrule
%		\end{tabular}
%		\caption{Average and standard deviation for the errors in the 
%			approximations made for the projections of points generated 
%			by elements of $\mathcal{R}_3$, using our approach when 
%			$N = 1000$. We also provide the average computational time 
%			by computation in milliseconds.}
%	\end{table}
\end{test}
Finally, we provide an experiment for the projection onto a cone 
generated by the epigraph of the perspective of a hyperbolic 
penalty function, which cannot be computed by the available 
conic solvers.
% MOSEK \cite{mosek2022paper}, SCS \cite{scs2016}, ECOS 
%\cite{ECOS} or COSMO \cite{cosmo2021}.
\begin{test}
	In the context of Example~\ref{ex:MBF}, using \eqref{e:mbf_persp} we deduce that
	\begin{align*}
		\epi \widetilde{f}= \menge{(x,\eta,\delta)\in\R\times \RPP \times \R}{\frac{\eta x}{\eta-x}\le \delta \text{ and } x < \eta} \cup ~ ]-\infty,0] \times \{0\} \times \RP .
	\end{align*}
In order to generate our synthetic data, we set $\varepsilon = 
10^{-15}$ and consider
	\begin{equation}
		\mathcal{R}_4 = \menge{(x,\eta,\delta) \in \R^3 }{\varepsilon \le \eta \le 100,-100 \le x < \eta \:\:\text{and}\:\: \delta = \frac{\eta x}{\eta-x}},
	\end{equation}
which is a 
subset of the boundary of $\epi \widetilde{f}$.
We randomly generate $\{(\widehat{x_i},\widehat{\eta_i}, 
\widehat{\delta_i})\}_{i = 1}^N \in 
	\mathcal{R}_4$ with $N =10000$. Similarly to the previous tests, for every $i \in \{1,\ldots,N\}$, we randomly chose $t_i \in 
	\left]0,10\right]$ and set
	\begin{equation}
		(x_i,\eta_i, \delta_i) = (\widehat{x_i},\widehat{\eta_i}, \widehat{\delta_i}) + t_i \left(\frac{\widehat{\eta_i}^2}{(\widehat{\eta_i}-\widehat{x_i})^2},\frac{-\widehat{x_i}^2}{(\widehat{\eta_i}-\widehat{x_i})^2},-1\right).
	\end{equation}
	
	Noting that $\widetilde{f}$ is differentiable in the interior of its domain, we deduce form Lemma~\ref{l:normal} that,
	for every $i \in \{1,\ldots,N\}$, $(x_i,\eta_i, 
	\delta_i)\notin \epi \widetilde{f}$, and $ (\widehat{x_i},\widehat{\eta_i}, 
	\widehat{\delta_i}) = P_{\epi \widetilde{f}}(x_i,\eta_i, 
	\delta_i) \coloneqq \widehat{p_i}$. 
	
	Next, we approximate $\{\widehat{p_i}\}_{i=1}^N$ using our approach. Furthermore, for every $i \in \{1,\ldots,N\}$, we denote by 
	$\mu_i \in \RPP$ to the approximation of the unique solution of \eqref{e:mu_e_equal} for
	$(x_i,\eta_i, \delta_i)$ using \textit{scipy.optimize.root\_scalar} \cite{scipy} 
	with Brent's method \cite[Section~9.3]{press2007numerical} 
	with tolerance $\epsilon_3=10^{-12}$. Then, using 
	\eqref{e:mbf_proj} and \eqref{e:mbf_prox_persp}, for every $i \in 
	\{1,\ldots,N\}$, we define
	\begin{equation}
		\label{e:mbf_proj_approx}
		p_i = 
		\begin{cases}
			\big(\min\{0,x_i\},0,\delta_i\big),&\text{if}\:\: 
			\eta_i \le 0,\:x_i \le -\eta_i,\text{ and } 0 \le \delta_i;\\[2mm]
			\big(x_i,\eta_i,\delta_i\big),&\text{if}\:\: 
			\eta_i \ge 0,\:x_i < \eta_i, \text{ and } \frac{\eta_i x_i}{\eta_i- 
			x_i}\le \delta_i;\\[2mm]
			\big(\prox_{\mu_i \widetilde{f}}(x_i,\eta_i),\delta_i + 
			\widehat{\mu_i}\big),&\text{otherwise}.
		\end{cases}
	\end{equation}
	
	In Table~\ref{tab:mbf} we exhibit the average and standard 
	deviation of the errors $\left\{\norm{\widehat{p_i} - 
	p_i}_2\right\}_{i=1}^N$. We also provide the average 
	computational time  in milliseconds used for each approximation.
We observe a very high precision in reasonable computational 
time as in former tests. 
	\begin{table}[H]
		\centering
		\begin{tabular}{lc}
			Approach &  \eqref{e:mbf_proj_approx} \\
			\toprule
			Error Average       & 3.48e-12 \\
			Error St. Deviation & 2.27e-10 \\
			Av. comput. time [ms] &  21.56\\
			\bottomrule
		\end{tabular}
		\caption{Average and standard deviation for the errors and 
		average computational time in the approximations made for 
		the projections of points generated by elements of 
		$\mathcal{R}_4$, using our approach when $N = 10000$.}
\label{tab:mbf}
	\end{table}
	
\end{test}
\section{Conclusions}
\label{s:5}
In this paper we provide an efficient computation for the projection 
onto the epigraph of the perspective of any lower semicontinuous 
convex function defined in a real Hilbert space. Our approach 
relies in the resolution of two 
coupled scalar equations, which can be solved with high precision.
We implement our formula in the case of the exponential cone and 
the hyperbolic cone, and we compare our approach with a 
state-of-the-art software in python.

\textbf{Acknowledgments.} The work of Luis M. Brice\~no-Arias is 
supported by  Centro de Modelamiento Matem\'atico (CMM), 
FB210005, BASAL fund for centers of excellence, and grants 
FONDECYT 1230257 and MATHAmSud 24-MATH-17 
from ANID-Chile. The work of Crist\'obal Vivar-Vargas was 
supported by Universidad T\'ecnica Federico Santa Mar\'ia by the 
grant Programa de Incentivo a la Investigaci\'on Cient\'ifica (PIIC).

\end{document}